\documentclass[11pt]{article}
\usepackage{fullpage,amsfonts,latexsym,bbm}

\newcommand{\nocH}{{H}}
\newcommand{\cN}{\mathcal{N}}
\newcommand{\cR}{\mathcal{R}}
\newcommand{\cS}{\mathcal{S}}

\newcommand{\hA}{\hat{A}}
\newcommand{\hB}{\hat{B}}
\newcommand{\hC}{\hat{C}}

\newtheorem{theorem}{Theorem}
\newtheorem{proposition}{Proposition}
\newtheorem{lemma}{Lemma}
\newtheorem{corollary}{Corollary}

\newcommand{\RR}{\mathbb{R}}
\newcommand{\CC}{\mathbb{C}}

\newenvironment{proof}
   {\noindent {\bf Proof.}}{\hfill$\Box$}
\newenvironment{block}{\left[ \begin{array}}{\end{array}\right] }
\def\cD{{\cal D } }

\def\diag{\mbox{diag} }
\def\transpose{T}
\def\Galpha{{\hA } }
\def\Gbeta{{\hB\hB^* } }
\def\Ggamma{{\hC^*\hC } }

\newtheorem{remark}{Remark}[section]

\begin{document}
\title{Minimal symmetric Darlington synthesis}

\author{L. Baratchart\thanks{Laurent Baratchart and Martine Olivi are   with INRIA, BP 93,
06902 Sophia-Antipolis Cedex, FRANCE,
        {\tt\small \{Laurent.Baratchart@sophia.inria.fr\},
\{Martine.Olivi@sophia.inria.fr\}}},%
     ~  P. Enqvist\thanks{Per Enqvist is with the Division of Optimization and Systems Theory, KTH, Stockholm, Sweden
        {\tt\small \{penqvist@math.kth.se\}}},%
     ~   A. Gombani\thanks{Andrea Gombani is with ISIB-CNR, Corso Stati Uniti 4, 35127, Padova, Italy,
        {\tt\small \{gombani@isib.cnr.it\}}},
         M.
    Olivi$^{*}$
}

\maketitle

{\bf Abstract.} We consider the symmetric Darlington synthesis of a $p \times p$ rational symmetric Schur function $S$ with the constraint that the extension is of size $2p \times 2p$. Under the assumption that $S$ is strictly contractive in at least one point of the imaginary axis, 
we determine the minimal McMillan degree of the 
extension. In particular, we show that it is generically given by the number of zeros of odd 
multiplicity of $I_p-SS^{*}$. A constructive 
characterization of all such extensions is provided in terms of a symmetric realization of $S$ and 
of the outer spectral factor of $I_p-SS^{*}$. The authors's motivation for the problem stems from Surface Acoustic Wave filters where physical constraints 
on the electro-acoustic scattering matrix naturally raise 
this mathematical issue.

\section{Introduction}

The Darlington synthesis problem has a long history which goes back to the time when computers were not available and the synthesis of non-lossless circuits was a hard problem: the brilliance of the Darligton synthesis was that it reduced any synthesis problem to a lossless one.
In mathematical terms,  given a $(p\times p)$ Schur function $S$, say, in the right half-plane, the problem is to imbed  $S$  into a $(m+p)\times(m+p)$-inner function $\cS$ so that:
\begin{equation}
\label{extg}
\cS=\left(
\begin{array} {c c} S_{11} & S_{12}\\
                   S_{21} & S
\end{array}
\right),~~~~~~~~~~ \cS(i\omega)\cS^*(i\omega)=I_{m+p},~~~~\omega\in\RR.
\end{equation}

This problem was first studied by
Darlington in the case of a scalar rational $S$ \cite{Darlington}, was generalized to the matrix case
\cite{Belevitch}, and finally carried over to non-rational $S$
\cite{Arov,Dewilde}. We refer the reader to the nice surveys 
\cite{BM,Dewilde2} for further references and generalizations
({\it e.g.} to the non-stationary case). An imbedding of the form (\ref{extg})
will be called a \emph{Darlington synthesis } or \emph{inner extension},
or even sometimes a \emph{lossless extension} of $S$.

A Darlington synthesis exists provided that $S(i\omega)$ has constant rank
a.e., that the determinant of $I_p-S(i\omega)S^*(i\omega)$ (viewed as an operator
on its range) satisfies the Szeg\"o condition (see e.g. \cite{Garnett}), and that $S$ is
pseudo-continuable across the imaginary axis meaning that there is a
meromorphic function in the left half-plane whose nontangential limits on $i\RR$ agree
with $S(i\omega)$ a.e. \cite{Arov,Arov73,Dewilde,Dewilde2,DH}. If moreover
$S$ has the conjugate-symmetry, then $\cS$ may be chosen with this property.
When $S$ is rational the above conditions are fulfilled so that
a Darlington synthesis always exists. In addition, it can be
carried-out without augmenting the McMillan degree ({\it i.e.}
we may require in (\ref{extg}) that ${\rm deg}\cS={\rm deg}S$) and choosing $m=p$; this follows
easily from Fuhrmann's realization theory \cite{F81} and the
arguments in \cite{Dewilde2} or else from more direct computations 
carried out in \cite{AV,GR88}. In particular $\cS$ can be chosen rational, and also
to have real coefficients if $S$ does.

When $S$ is the scattering matrix of an electric $p$-pole without gyrators \cite{Belevitch},
the reciprocity law entails that $S$ is symmetric and the question arises 
whether the extension $\cS$ can also be made
symmetric; this would result in a Darlington synthesis which is itself free from
gyrators. In \cite{AV} it is shown that a symmetric Darlington synthesis of
a symmetric rational $S$ indeed exists and, although one can no longer
preserve the degree while keeping $m=p$, he can at least ensure that ${\rm deg \cS}\leq
2{\rm deg S}$. 
The existence of a  symmetric Darlington
synthesis for non-rational functions has been studied in \cite{Arov73},
in the slightly different but equivalent setting of $J$-inner extensions.

In \cite{AV} it is also shown that, by increasing the \emph{size} $m$ to 
$p +n$, where $n$ is the degree of $S$, it is possible to construct a 
{\em symmetric} extension of exact degree $n$.
However, such an increase of $m$ is not always appropriate.
In fact, although the original motivations from circuit synthesis that brought
the problem of lossless imbedding to the fore are mostly forgotten today, 
the authors of the present paper were led to raise the above issue in
connection with the modeling of Surface Acoustic Waves filters
\cite{BEGO}. In this context, physical constraints impose $m=p$, so that each block of the
electro-acoustic scattering matrix $\cS$ in (\ref{extg}) has to be of size $p \times p$.

It is thus natural to ask  the following : \emph{given a
  symmetric rational $S$, what is the minimal degree of a symmetric 
lossless extension $\cS$?} This is the problem that we consider.
For \emph{scalar} systems, this minimal degree has been known for a while and can be found, for instance in
\cite{YWP} (see also Section 3 below). The present paper will generalize this to the matrix-valued case.
We restrict our attention to the case where $S$ is strictly contractive in at least 
one point of the imaginary axis. This implies that the extension will  have size $2p$. For the general case, that is, with extensions of lower size, the analysis 
seems to be more difficult and it will possibly be treated in a subsequent paper.

In Section \ref{prelim} we introduce some notations. In Section \ref{scalarcase}, we 
shed light on the problem by discussing the elementary \emph{scalar case}, that is, $p=1$. In
Section \ref{innercomp} we recall some results of Gohberg and Rubinstein \cite{GR88} about a state space 
construction of an inner extension preserving the degree and we characterize all inner extensions in terms of minimal ones. In Section \ref{symun} we present a simple method to construct (possibly unstable) symmetric unitary extensions. 
In Section \ref{minsym} we finally produce a symmetric inner extension of minimal degree. 
In Section \ref{realcase} we discuss the symmetric \emph{and conjugate symmetric}
 unitary extension of a rational 
symmetric  Schur function which is conjugate symmetric ({\it i.e.} that has
real coefficients), and we show on an example that its 
minimal degree is generally \emph{larger} than the one attainable without the conjugate-symmetry 
requirement.

\section{Preliminaries and notations}\label{prelim}

Throughout, if $M$ is a complex matrix,
we let ${\rm Tr}(M)$ stand for
its trace, $M^{T}$ for its transpose, and
$M^{*}$ for its transpose-conjugate.
We denote respectively by 
\[\Pi^+=\{s\in\CC;~{\rm Re\;} s>0\}~~~~{\rm and}~~~~\Pi^-=\{s\in\CC;~{\rm Re\;} s<0\}\] 
the right and left half-planes.

In System Theory, a rational function whose poles lie
in $\Pi^-$ is called \emph{stable},
and a rational function which is finite (resp. vanishing) at infinity is
called \emph{proper} (resp. \emph{strictly proper}).
System Theory is often concerned with
functions having the \emph{conjugate symmetry}:
$W(\bar{s})=\overline{W(s)}$, but we shall not make this
restriction unless otherwise stated. A rational
function has the conjugate-symmetry if, and only if, it has real coefficients.

For $W(s)$ a matrix-valued function on $\CC$, 
we define its para-hermitian conjugate $W^*$ to be:
\begin{equation}
\label{para-hermitian}
W^{*}(s):= W(-\bar{s})^*.
\end{equation}
\emph{Note that $*$ has two different meanings depending on its position with respect to the variable; this slight ambiguity is common in the literature and allows for a simpler notation}.

Note that  $W^*(i\omega)=W(i\omega)^*$ on the imaginary axis, and 
if $W$ is a polynomial then
$W^*$ is also a polynomial whose zeros are reflected 
from those of $W$ across the imaginary axis.

We say that a rational
$p \times m$
matrix-valued function $S$ holomorphic on $\Pi^+$ is a \emph{Schur function} 
if it is contractive:
\begin{equation}
\label{contdef}
S(s)S(s)^*\leq I_p, ~~s\in\Pi^+.
\end{equation}

A rational $p\times p$ Schur function $S$ is said to be \emph{lossless}, or \emph{inner},
if 
\begin{equation}
\label{losslessdef}
S(i\omega)S(i\omega)^*=I_p,
\hspace{1cm} \omega \in \RR.
\end{equation}
The scalar rational inner functions are of the
form $q^*/q$ where $q$ is a
polynomial whose roots lie in $\Pi^-$; if ${\rm deg}\,q=n$, 
such a function is called a \emph{Blaschke product of degree} $n$.
A (normalized) Blaschke product of degree 1 is just a M\"obius transform of the type
\begin{equation}
\label{defbxi}
b_{\xi}(s) :=(s-\xi)/(s + \overline{\xi}),~~~~~~\xi \in\Pi^+.
\end{equation}
The natural extension to the matrix case is given by
\begin{equation}
\label{defBxi}
B_{\xi} (s)= \left( \begin{array}{cc}
b_{\xi}(s) & 0  \\
0             & I_{p-1} 
\end{array} \right) ,~~~~\xi\in\Pi^+,
\end{equation}

which is the most elementary example of an inner function  of degree 1. Actually, 
it is a result of Potapov \cite{Pot60,Dym} that
these and unitary matrices together generate all rational inner matrices. More precisely, if
$Q$ is such a matrix and $\xi_1,\ldots,\xi_n$ its zeros 
({\it i.e.}  the zeros of its determinant which is a Blaschke product)
ordered arbitrarily counting multiplicities, there exist complex unitary matrices 
$U_1,\ldots,U_{n+1}$ such that
\begin{equation}
\label{Potapov}
Q=U_1\,B_{\xi_1}\,U_2\,B_{\xi_2}\ldots U_n\,B_{\xi_n}\,U_{n+1}.
\end{equation}
An inner matrix like $U_1\,B_{\xi_1}\,U_2$ is often called an \emph{elementary Blaschke factor}.

Given a proper rational matrix $S$, we shall write
\begin{equation}
\label{notR}
 S = \left( \begin{array}{c|c} A & B \\ \hline C & D 
\end{array}\right)
\end{equation}
whenever $(A,B,C,D)$ is a \emph{realization} of 
$S$, in other words whenever $S(s) = C(sI_n-A)^{-1}B+D$ where $A,B,C,D$ 
are complex
matrices of appropriate sizes. Because $S$ in this case
is the so-called transfer function \cite{AV,BGK,Kalman} of the linear dynamical system:
\begin{equation}
\label{sysdyn}
 \dot{x}=Ax+Bu,~~~~y=Cx+Du,\end{equation}
with state $x$, input $u$, and output $y$, we say sometimes that $A$ is a 
\emph{dynamics matrix} for $S$. The matrices $C$ and $B$ are respectively  called
the \emph{output} and \emph{input} matrices of the realization. 

Every proper rational matrix has infinitely many realizations,
and a realization is called \emph{minimal} if $A$ has minimal size.
This minimal size will be taken as definition of the \emph{McMillan 
degree} of $S$, abbreviated as ${\rm deg} S$.
As is well-known \cite{Kalman,BGK}, the realization (\ref{notR})
is minimal if and only if \emph{Kalman's criterion} 
is satisfied, that is if the
two matrices:
\begin{equation}
\label{Kalman}
\begin{block}{cccc}
B & AB & \dots & A^{n-1}B 
\end{block}, \quad
\begin{block}{cccc}
C^T & A^TC^T & \dots & (A^T)^{n-1}C^T 
\end{block},
\end{equation}
are surjective, where $n$ denotes the size of $A$. 

The surjectivity of the 
first matrix expresses the \emph{reachability} of the system, and 
that of the second matrix  its
\emph{observability}. Any two minimal realizations can be
deduced from each other by a linear change of coordinates:
\[(A,B,C,D)\mapsto(TAT^{-1},TB,CT^{-1},D),~~~~T ~
\mathrm{an~invertible~matrix},\]
so that a dynamics matrix of minimal size for $S$ is well-defined
up to similarity.  In particular the eigenvalues of $A$ depend only on $S$
and they are in fact its poles, the \emph{multiplicity} of a pole being its total multiplicity as an 
eigenvalue by definition. The sizes of the Jordan blocks associated to an eigenvalue
are called the \emph{partial multiplicities} of that eigenvalue.
The partial multiplicities may be computed as follows.
Performing elementary row and column operations on $S$, 
one can put it in {\em local Smith form} at $\xi$ (see {\it e.g.} \cite[sec.7.2.]{GLR}
or \cite{BO2,BGR}):
\begin{equation}
\label{localSmithform}
S(s)=E(s)\diag[(s-\xi)^{\nu_1},(s-\xi)^{\nu_2},\ldots,(s-\xi)^{\nu_k},0,\ldots,0]F(s)
\end{equation}
where $E(s)$ and $F(s)$ are rational matrix functions that are finite and
invertible at $\xi$ while $k$ is the rank of $S$ as a rational matrix and
$\nu_1\leq \nu_2\leq \ldots\leq\nu_k$ are relative integers. These
integers  are uniquely determined by $S$ and sometimes called its {\em partial
  multiplicities} at $\xi$. Note that $\xi$ is a {\em pole} 
if, and only if, there is
at least one negative partial multiplicity at $\xi$. In fact, 
the negative partial multiplicities at $\xi$ are precisely the 
partial multiplicities of $\xi$ as a pole of $S$. 

One says that $\xi$ is a \emph{zero} of $S$ if the local Smith forms exhibits 
at least one positive partial multiplicity at $\xi$, and the positive partial multiplicities at $\xi$
are by definition the partial multiplicities of $\xi$ as a zero.
If $S$ is invertible as a rational matrix, 
it is clear from (\ref{localSmithform}) that the poles of $S^{-1}$
are the zeros of $S$, with corresponding partial multiplicities.
Note also that a zero may well be at the same time a pole, which causes many of the difficulties in 
the analysis of matrix valued functions.
When $S$ is inner, which is our main concern here, this does not happen
because its poles lie in $\Pi^-$ and its
zeros in $\Pi^+$.

A rational matrix has real
coefficients if, and only if, there exists a minimal realization 
which is real, {\it i.e.} such that $A$, $B$, $C$, and $D$ are real matrices.
As is customary in System Theory, we
occasionally refer to a proper rational matrix as being a \emph{transfer function}.
If it happens to have the conjugate-symmetry, we say it is a \emph{real
  transfer function}. 

The system-theoretic interpretation (\ref{sysdyn}) of
(\ref{notR}) makes it easy to compute a realization for a product of
transfer-functions. In fact, if $S_1$ is $m\times k$ and $S_2$ is $k\times 
p$, and if
\begin{equation}
\label{realdiv}
 S_1 = \left( \begin{array}{c|c} A_1 & B_1 \\ \hline C_1 & D_1 
\end{array}\right),~~~~~~~~
 S_2 = \left( \begin{array}{c|c} A_2 & B_2 \\ \hline C_2 & D_2 
\end{array}\right),
\end{equation}
then a short computation shows that the following two realizations hold:
\begin{equation}
\label{Rprod}
 S_2S_1 = 
\left( \begin{array}{cc|c} A_1 &0& B_1\\
 B_2C_1  & A_2 & B_2D_1 \\ \hline  D_2C_1 & C_2 & D_2D_1 
\end{array}\right),~~~~~~
 S_2S_1 = 
\left( \begin{array}{cc|c} A_2 &B_2C_1& B_2D_1\\
0  & A_1 & B_1 \\ \hline  C_2 & D_2C_1 & D_2D_1 
\end{array}\right).
\end{equation}
Likewise, if $k=m$ and $D_1$ is invertible (so that $S_1$ is {\it a fortiori}
invertible as a rational matrix), then
\begin{equation}
\label{realinv}
 S_1^{-1} = \left( \begin{array}{c|c} A_1-B_1D_1^{-1}C_1 & B_1D_1^{-1} \\ \hline -D_1^{-1}C_1 & D_1^{-1} 
\end{array}\right).
\end{equation}
Using (\ref{Kalman}), it is immediate that the minimality of
(\ref{realdiv}) implies that of (\ref{realinv}). In contrast,
the realizations (\ref{Rprod}) need not be minimal even if the
realizations (\ref{realdiv}) are: pole-zero cancellations may occur in the 
product $S_2S_1$ to the effect that the multiplicity of a pole may not be
the sum of its multiplicity as a pole of $S_2$  ({\it i.e.} an eigenvalue of $A_2$)
and as pole of $S_1$ ({\it i.e.} an eigenvalues of $A_1$). One instance where (\ref{Rprod})
is minimal occurs when $S_1$, $S_2$ have full rank and
no zero of $S_1$ is a pole of $S_2$ and no zero of $S_2$ is a pole of $S_1$
(see \cite{CN}). 
This, in 
particular, is satisfied when $S_1$ and $S_2$ are inner, 
implying that the McMillan degree of a 
product of inner functions is 
the sum of the McMillan degrees.
 By (\ref{Potapov}), this in turn implies that the McMillan 
degree of an inner function is also the degree of its determinant viewed as a scalar Blaschke product.

Whenever $S$ is a transfer function, its transpose $S^T$ clearly has the same 
McMillan degree as $S$. 
A square transfer function $S$ is called \emph{symmetric} if $S=S^T$, and
then a realization is called \emph{symmetric} if 
$A=A^\transpose$, $B^\transpose=C$ and $D=D^\transpose$. 
It is not too difficult to see that a transfer function is symmetric if, and
only if, it has a minimal realization which is symmetric \cite{FH95}. The
latter \emph{may be complex} even if $S$ is a real transfer function.

\section{The case of a scalar Schur function.}\label{scalarcase}
For getting an idea of the solution to our
problem, we first consider 
the symmetric inner extension of a scalar rational Schur function 
to a $2\times2$ inner rational function,  
that is we assume momentarily that $p=m=1$. This case has been considered in \cite{YWP}. Put
\[  S = \frac{p_1}{q}, \]
where $p_1$ and $q$ are coprime polynomials such that 
${\rm deg}\{ p_1\}\leq {\rm deg}\{ q\}$, $p_1$ is not identically zero,
$|p_1(i\omega)|\leq|q(i\omega)|$ for $\omega\in\RR$,
and $q$ has roots in the open left half-plane only. 
The McMillan degree of $S$ is just the degree of $q$ in this case. 
As the orthogonal space to a nonzero vector $v=(a,b)^T\in\CC^2$ is spanned 
by $(-\bar{b},\bar{a})^T$, it is easily checked that
every rational inner extension $\cS$ of $S$, when all
its entries are written over a common denominator, say,
$dq$ where $d$ is a stable monic polynomial, is of the form
 \[{\cal S}  = 
\frac{1}{dq}  
\begin{block}{cc} e^{i\theta_1} & 0 \\ 0 & 1 \end{block}
\begin{block}{cc} 
(dp_1)^* & -p_2^* \\ p_2 & dp_1 
\end{block}
\begin{block}{cc} e^{i\theta_2} & 0 \\ 0 & 1 \end{block}
 \]
where  $\theta_1,\theta_2\in \RR$ and $p_2$ is a polynomial solution of degree
at most $\mathrm{deg}\{ dq\}$  to the 
spectral factorization problem:
\begin{equation}
\label{specfact}
dp_1(dp_1)^* + p_2p_2^* = dq (dq)^*,
\end{equation}
whose solvability is ensured by the contractivity of $S$.
Clearly the extension is symmetric if and only if 
$- e^{i\theta_1}p_2^*= e^{i\theta_2}p_2$, which is compatible with
(\ref{specfact}) if, and only if, 
all zeros of the polynomial 
\[ dq (dq)^*-dp_1(dp_1)^* =dd^*(qq^*-p_1p_1^*)
\]
have even multiplicity. Consider the polynomial
\begin{equation}
\label{defmu}
 \mu:= qq^*-p_1p_1^*,
\end{equation}
and single out its roots of even multiplicity by
writing $\mu=(r_1r_1^*)^2r_2r_2^*$, where $r_1$ and $r_2$ are stable coprime polynomials and 
all the roots of $r_2$ are simple. For $dd^*\mu$ to have roots of even multiplicity only,
it is then necessary that $r_2$ divides $d$. Therefore, 
as the McMillan degree of an inner function is the degree of its determinant,
we get
\begin{equation}\label{scalardeg}
\mathrm{deg}\{ \cS\}=\mathrm{deg}\{ dq\}\geq\mathrm{deg}\{ r_2q\}.
\end{equation}
On another hand, a symmetric inner extension of McMillan degree $\mathrm{deg}\{ r_2q\}$
is explicitly given by
\begin{equation}\label{scalarext}
\cS_m=\begin{block}{cc}
-\frac{p_1^*}{q}\frac{r_2^*}{r_2} & 
\frac{r_1r_1^* r_2^*}{q} \\[1ex]
\frac{r_1r_1^* r_2^*}{q} & \frac{p_1}{q}
\end{block}
=
\begin{block}{cc}
-\frac{p_1^*}{q} &
\frac{r_1r_1^* r_2^*}{q} \\[1ex]
\frac{r_1r_1^* r_2}{q} & \frac{p_1}{q}
\end{block}
\begin{block}{cc}
\frac{r_2^*}{r_2} & 0 \\[1ex] 
0 & 1
\end{block}.
\end{equation}

Thus we see already in the scalar case that the minimal attainable degree for a symmetric
inner extension of $S$ is the degree of $S$ augmented by half the number of zeros of
$\mu$ of odd multiplicity.
Formulas (\ref{scalardeg}) and (\ref{scalarext}) should be compared with the corresponding formulas (98) and (95) in \cite{YWP}.

As $1-SS^*=\mu/(qq^*)$, the zeros of
$\mu$ are the zeros of $1-SS^*$ augmented by the common zeros to
$p_1$ and $q^*$ and the common zeros to $p_1^*$ and $q$;
the latter of course are reflected from the former across the imaginary axis,
counting multiplicities.
In particular a degree-preserving symmetric
Darlington synthesis requires special conditions that can be rephrased  as:
\begin{itemize}
\item[(i)] the zeros of $1-SS^*$ have even multiplicities,
\item[(ii)] each common zero to $S$ and $(S^*)^{-1}$, if any,
 is common with even multiplicity.
\end{itemize}

\begin{remark}
Note that \rm{(i)} is automatically fulfilled for those zeros located \emph{on} the
imaginary axis, if any, so the condition really bears on the 
non-purely imaginary zeros. Note also that \rm{(ii)} concerns those zeros of $S$,
if any, whose reflection across the imaginary axis is a pole of $S$; by the
coprimeness of $p_1$ and $q$, such zeros are never purely imaginary.
\end{remark}

Our goal is to generalize the previous result to matrix-valued contractive rational
functions.

\section{Inner extensions.}
\label{innercomp}
We shall restrict our study to the case where the function $S$ to be
imbedded is strictly contractive at infinity: $\|S(\infty)\|<1$. 
If $S$ is strictly contractive at some finite point $i\omega_0$, the
change of variable $s\to 1/(s-i\omega_0)$
will make it contractive at infinity and such 
a transformation preserves rationality and the McMillan degree while mapping
$\Pi^+$ onto itself, hence our results carry over immediately
to this case. But if $S$ is strictly contractive \emph{at no point} of the 
imaginary axis, then our method of proof runs into difficulties and the 
answer to the minimal degree symmetric inner extension issue will remain open.
To recap, we pose the following problem:

\emph{Given a $p\times p$ symmetric rational
Schur function which is strictly contractive at infinity,
what is the minimal McMillan degree of a $2p\times 2p$ 
inner extension $\cal S$ of $S$ which is also symmetric~:}
\begin{equation}
\label{extsymg}
\cS=\left(
\begin{array} {c c} S_{11} & S_{12}\\
                   S_{21} & S
\end{array}
\right),~~~~ S_{11}=S_{11}^T,~~S_{21}=S_{12}^T,
~~\cS(i\omega)\cS^*(i\omega)=I_{2p},~~\omega\in\RR.
\end{equation}

\subsection{Inner extensions of the same McMillan degree.}

Our point of departure will be the solution to the Darlington synthesis problem for a rational function in terms of realizations. Let $S$ be a  Schur $p\times p$ function which is strictly contractive at infinity and let
\begin{equation}
\label{real:S}
 S = \left( \begin{array}{c|c} A & B \\ \hline C & D 
\end{array}\right)
\end{equation}
be a minimal realization of $S$ of degree $n$. 
The strict contractivity at infinity means that
$I_p-D^* D$ and $I_p-DD^*$ are positive definite. Therefore
we may set
\begin{eqnarray}
\Galpha & = & A+BD^*(I_p-DD^*)^{-1}C, \label{def:Galpha}\\
\hB & = & B(I_p-D^*D)^{-1/2}, \label{def:Gbeta}\\
\hC & = & (I_p-DD^*)^{-1/2}C, \label{def:Ggamma}
\end{eqnarray}
and subsequently we define:
\begin{equation} \label{def:A}
\nocH = \begin{block}{cc} -\Galpha^* & -\Ggamma \\ \Gbeta&\Galpha \end{block}. 
\end{equation}

\begin{lemma}\label{Adyn}
Assuming $S$  is a Schur function 
strictly contractive at infinity given by (\ref{real:S}),
then the matrix $\nocH$ defined in (\ref{def:A})
is a dynamics matrix of $(I_p-SS^*)^{-1}$.
Furthermore $\nocH$ is Hamiltonian, i.e. 
\[\nocH^{*}  \left[ \begin{array}{cc}
    0 & I_{p} \\ 
    -I_{p} & 0 \\ 
  \end{array} \right]=-  \left[ \begin{array}{cc}
    0 & I_{p} \\ 
    -I_{p} & 0 \\ 
  \end{array} \right]\nocH.
\]
\end{lemma}
\begin{proof}
By definition
\[ S^* = \left( \begin{array}{c|c} -A^* & -C^* \\ \hline B^* & D^* 
\end{array}\right),\]
then, from (\ref{Rprod})
\[ I_p-SS^* = \left( \begin{array}{cc|c} -A^* &0& C^*\\
 BB^*  & A & -BD^* \\ \hline  DB^* & C & I_p-DD^* 
\end{array}\right),\]
and if $S$ is strictly contractive at infinity 
the inverse of $I_p-DD^*$ is well defined. Then from (\ref{realinv}), we have
\[ (I_p-SS^*)^{-1} =
\left( \begin{array}{cc|c}  -A^*-C^*\Delta_l DB^*& -C^*\Delta_l C &
 C^*\Delta_l\\
B\Delta_r B^* & A+BD^*\Delta_lC &  -BD^*\Delta_l\\ \hline 
 -\Delta_l DB^* & -\Delta_lC & \Delta_l
\end{array}\right),\]
where $\Delta_l=(I_p-DD^*)^{-1}$ and  $\Delta_r=(I_p-D^*D)^{-1}$,
whose dynamics matrix is none but $\nocH$.

Finally, it is easy to check from the definitions
(\ref{def:Galpha})-(\ref{def:Ggamma}) that $\nocH$ is a Hamiltonian matrix, 
i.e. that the 
partition of $\nocH$ defined in (\ref{def:A}) satisfies 
$\nocH_{12}^*=\nocH_{12}$, 
$\nocH_{21}^*=\nocH_{21}$, and $\nocH_{22}^*=-\nocH_{11}$.
\end{proof}

\begin{remark}
The Hamiltonian character of $\nocH$ implies that it is
similar to $-\nocH^*$. In particular the eigenvalues of $\nocH$ are
symmetric with respect to the imaginary axis, counting multiplicities.
It must also be stressed that the realization of 
$(I_p-SS^*)^{-1}$ given in the
proof of Lemma \ref{Adyn} may not be minimal. Because the McMillan
degree is invariant upon taking the inverse, the realization in question will
in fact be minimal if, and only if,
the McMillan degree of $SS^*$ is the sum of the McMillan
degrees of $S$ and $S^*$. This will hold in particular
when no zero of $S$ is a pole of $S^*$ \cite{BGK}, 
in other words if no zero of $S$
is reflected from one of its poles.
Hence the characteristic polynomial of
$\nocH$ plays in the matrix-valued case the role of the polynomial $\mu$ 
given by (\ref{defmu}) in the scalar case (compare condition (ii) after
(\ref{defmu})).
\end{remark}

The (not necessarily  symmetric) inner extensions of $S$ that preserve the McMillan degree
are characterized by the following theorem 
borrowed from \cite{GR88}. Actually, theorem 4.1 in \cite{GR88} describes all the
rational unitary (on the real line) extensions of a (non necessarily square) rational matrix function 
which is contractive on the real
line and strictly contractive at infinity.  The next theorem essentially rephrases
this result in our right half plane setting dealing with
square and stable matrix functions.

\begin{theorem} \label{thm:GR}
If $S$ given by (\ref{real:S}) is a Schur function which is
strictly contractive at infinity,
then all $(2p)\times(2p)$ inner extensions 
\begin{equation}
\label{extinmd}
\cS=\left(
\begin{array} {c c} S_{11} & S_{12}\\
                   S_{21} & S
\end{array}
\right)
\end{equation}
of the same McMillan degree as $S$ are given by
\begin{equation}\label{Wext} 
 {\cal S} = 
\begin{block}{cc} U_2 & 0 \\ 0 & I_p \end{block}
{\cal S}_P
\begin{block}{cc} U_1 & 0 \\ 0 & I_p \end{block}
\end{equation}
where $U_1$ and $U_2$ are arbitrary unitary matrices and where
${\cal S}_P$ is given by
\begin{equation}\label{Ireal}
 {\cal S}_P  = 
\left( \begin{array}{c|cc} A & B_1 & B \\ \hline 
C_1 & D_{11} & D_{12} \\
C & D_{21} & D 
\end{array}\right),
\end{equation}
with
\begin{eqnarray}
\label{d21} 
  D_{21}=(I_p-DD^*)^{1/2}, \quad  
\label{d12} 
  D_{12}=(I_p-D^*D)^{1/2}, \quad  
\label{d11} 
  D_{11}=-D^*, 
\end{eqnarray}
\begin{equation}\label{c}
 C_1 = - (I_p-D^*D)^{-1/2}(B^*P^{-1}+D^*C),
\end{equation}
\begin{equation}\label{b} 
B_1 = - (PC^*+BD^*) (I_p-DD^*)^{-1/2},
\end{equation}
and $P$ is a Hermitian solution to the algebraic Riccati equation:
\begin{equation}\label{RE} 
\cR(P) = P\Ggamma P + \Galpha P + P\Galpha^* +\Gbeta =0. 
\end{equation}
The map $P\to \cS_P$ is a one-to-one correspondence between the Hermitian
solutions to (\ref{RE}) and the inner extensions of degree $n$ of $S$ whose
value at infinity is $\cD$ defined in (\ref{Ireal}).
\end{theorem}

\begin{remark}\label{Pinv}
Note that \cite[thm. 3.4.]{GR88} guarantees,
under the assumptions of Theorem \ref{thm:GR}, that
all Hermitian solutions of (\ref{RE}) are invertible and positive definite
since $S(s)$ is stable.
\end{remark}

\subsection{Relation to spectral factors.}
\label{relspec}
We say that a stable $p\times p$ matrix-valued function $S_L$ (resp. $S_R$) is a left (resp. right)
spectral factor of $I_p-SS^{*}$ (resp. $I_p-S^{*}S$) if $S_LS_L^*+SS^{*}=I_p$ (resp.
$S_R^*S_R+S^*S=I_p$); such a factor is called minimal if it is rational and if
the block rational matrix $(S_L~S)$ (resp. $(S_R^T~S^T)^T$), whose McMillan degree is at least the degree
of $S$, actually has the \emph{same} McMillan degree as
$S$. This is equivalent to require that $(S_L~S)$ (resp. $(S_R^T~S^T)^T$) has a minimal realization whose
output (resp. input) and dynamics matrices are those of a minimal realization of $S$.
In particular Theorem \ref{thm:GR} implies that, for any inner extension of $S$ having the same 
McMillan degree, $S_{21}$ (resp. $S_{12}$) is a minimal left (resp. right) spectral
factor of $I_p-SS^{*}$ (resp. $I_p-S^{*}S$).

The corollary below is essentially a rephrasing of the theorem in terms of minimal spectral factors,
compare \cite{A73,F95,W}.
Observe that substituting (\ref{d21})-(\ref{b}) in (\ref{Ireal}) 
yields \[  \left[ \begin{array}{c}
    C_1 \\ 
    C \\ 
  \end{array} \right]=-\left[ \begin{array}{cc} 
 D_{11} & D_{12} \\
D_{21} & D 
\end{array}\right]  \left[ \begin{array}{c}
    B_1^* \\ 
    B^* \\ 
  \end{array} \right]P^{-1}\]
 while a similar substitution in (\ref{RE}) yields
\begin{equation}\label{lyapeq}
AP+PA^{*}+B_1B_1^{*}+BB^{*}=0.
\end{equation}
Note that (\ref{lyapeq}) has a unique solution since $A$ has no purely imaginary eigenvalue,
and that this solution is necessarily Hermitian positive definite by the controllability of $[A~B]$.
\begin{corollary}\label{innerext}
Let $S=  \left[ \begin{array}{c|c}
  A   &  B \\ 
\hline\rule{0cm}{.42cm}
   C  & D \\ 
  \end{array} \right]$ be a minimal realization of a Schur function strictly contractive
at infinity, and define $D_{11}$, $D_{21}$, and $D_{12}$ as in (\ref{d12}).
To each  
minimal left spectral factor $S_{21}$ of $I_p-SS^{*}$ with value
$D_{21}$ at infinity,  there is a unique inner extension of $S$ of the same McMillan degree, whose lower left block is $S_{21}$ and ,
with value  at infinity:
\[\label{scriptD}
\cD=\left[ \begin{array}{cc} 
 D_{11} & D_{12} \\
D_{21} & D 
\end{array}\right]
\]
If we put 
\begin{equation}
\label{redresse}
[S_{21}~ S]=  \left[ \begin{array}{c|cc}
  A   & B_1 & B \\ 
\hline\rule{0cm}{.42cm}
   C  & D_{21} & D \\ 
  \end{array} \right],
\end{equation}
then this extension is none but 
$\cS_{P}$ given by (\ref{Ireal}) where $P$ is the unique solution to (\ref{lyapeq}).
Alternatively, one also has
\begin{equation}
\label{eqjnsp}
\cS_{P}=\left[ \begin{array}{c|cc}
  A   & B_1 & B \\ 
\hline\rule{0cm}{.42cm}
-(D_{11}B_1^{*}+D_{12}B^{*})P^{-1} & D_{11} & D_{12}\\
   C  & D_{21} & D \\ 
  \end{array} \right]
\end{equation}
\end{corollary}
\begin{proof}
Suppose we have an inner extension $\cS$ of $S$, with the stated properties, given by (\ref{extinmd}).
Then, from (\ref{lyapeq})-(\ref{redresse}),
the block $S_{21}$ uniquely defines
$P$ thus also $\cS$ by Theorem \ref{thm:GR}.
Conversely, let $S_{21}$ be a minimal left spectral factor of $I_p-SS^*$ with value $D_{21}$ at infinity. 
Then we have  a realization of the form 
(\ref{redresse}) for $[S_{21}\ S]$ and we may define $P$ through
(\ref{lyapeq}). Using (\ref{Rprod}), we obtain for $I_p-[S_{12}~S][S_{12}~S]^*$ the following realization
\[ I_p-[S_{21}~S][S_{21}~S]^* = \left( \begin{array}{cc|c} -A^* &0& C^*\\
 B_1B_1^*+BB^*  & A &-B_1D_{21}^* -BD^* \\ \hline  D_{21}B_1^*+DB^* & C & I_p-D_{21}D_{21}^*-DD^* 
\end{array}\right).\]
Performing the change of basis defined by $\left[ \begin{array}{c|c}
  I_p  &  0 \\ 
\hline\rule{0cm}{.42cm}
   -P  & I_p \\ 
  \end{array} \right]$ using (\ref{d12}) and (\ref{lyapeq}), we find another realization to be
\[ I_p-[S_{21}~S][S_{21}~S]^* = \left( \begin{array}{cc|c} -A^* &0& C^*\\
 0  & A &-PC^*-B_1D_{21}^* -BD^* \\ \hline  CP+D_{21}B_1^*+DB^* & C & 0 
\end{array}\right).\]
But the rational function under consideration is identically zero by definition of $S_{21}$, hence 
by the observability of $[C~A]$ we get in particular
$-PC^*-B_1D_{21}^* -BD^*=0$ which yields
\begin{equation}
\label{calcC}
C=-[D_{21}~D][B_1~B]^*P^{-1}.
\end{equation}
 Now, put $C_1=-[D_{11}~D_{12}][B_1~B]^*P^{-1}$ and let
$\cS$ be defined as in (\ref{extinmd}).
Starting from the realization of $I_{2p}-\cS\cS^*$ provided by (\ref{Rprod}) and performing the 
change of basis defined by $\left[ \begin{array}{c|c}
  I_p   &  0 \\ 
\hline\rule{0cm}{.42cm}
   -P  & I_p \\ 
  \end{array} \right]$, a computation entirely similar to the previous one shows that this is the zero 
transfer function, that is, $\cS$ is inner. Moreover, (\ref{calcC}) shows it is an extension of $S$
whose lower left block is $S_{21}$, and clearly it has the same McMillan degree 
and value $\cD$ at infinity.  Finally, it is straightforward to check that this extension 
satisfies (\ref{eqjnsp}).
\end{proof}

 An inner extension $\cS_P$ of $S$ with the same McMillan degree and value
$\cD$ at infinity is thus
completely determined by the  choice of a minimal left spectral factor of
$I_p-SS^*$.  
Of course a dual result holds true 
on the right, namely the extension is also uniquely determined by a right minimal spectral factor 
$S_{12}$ of $I_p-S^*S$. In what follows, we only deal with inner extensions
having value $\cD$ at infinity, which the normalization induced by (\ref{Wext})
on letting $U_1=U_2=I_p$ there.

Let $P$ be a solution of the Riccati equation (\ref{RE}). Then the  
matrix  $\nocH$  defined in (\ref{def:A})  satisfies the similarity relation
\begin{equation}
\label{simcalA}
\left[
\begin{array} {cc} I_p & 0\\
                            -P & I_p
\end{array}\right]\nocH
\left[
\begin{array} {cc} I_p & 0\\
                            P & I_p
\end{array}\right]
=\left[
\begin{array} {cc} -(\Galpha+P\Ggamma)^* & -\Ggamma \\
0 & \Galpha+P\Ggamma \end{array}\right].
\end{equation}
Thus
\begin{equation}
\label{chiAsplit}
\chi_{\nocH}(s)=\chi_{Z}(s) \chi_{-Z^*}(s),
\end{equation}
where $\chi_M$ denotes the characteristic polynomials of $M$ and where we have set 
\begin{equation}
\label{Z}
Z=\Galpha+P\Ggamma.
\end{equation}
Moreover,  since $S$ is strictly contractive at infinity, $S_{21}$ is invertible and, in view of (\ref{realinv}), $S_{21}^{-1}$ has the dynamics matrix
\begin{equation}
A-B_1D_{21}^{-1}C = A+BD^*(I_p-DD^*)^{-1}C+PC^*(I_p-DD^*)^{-1}C 
= \Galpha + P \Ggamma=Z. 
\label{dynZ}
\end{equation}
Likewise, $S_{12}^{-1}$ has the dynamics matrix
\begin{equation}
\label{dynZ*}
A-BD_{12}^{-1}C_1=-P(\Galpha+P\Ggamma)^*P^{-1}=-PZ^*P^{-1}.
\end{equation}
This way the  extension process is seen to divide out the eigenvalues of $\nocH$ between
the inverses of the left and right spectral factors of $I_p-SS^*$ and $I_p-S^*S$ respectively.

It is a classical fact that there exists a natural partial ordering on the set of Hermitian solutions to
(\ref{RE}), namely $P_1\leq P_2$ if and only if the difference $P_2-P_1$ is
positive semidefinite. 
It is well-known (see \cite[sect.2.5.]{LR1} or \cite{LR2})  that there exists a maximal
solution $\hat P$ and a minimal solution $\check 
P$ of (\ref{RE}): $\hat P$ is the unique solution for which
$\sigma(\Galpha+\hat P\Ggamma)\subset \overline\Pi^+$, while $\check P$ is the unique solution for which
$\sigma(\Galpha+\check P\Ggamma)\subset \overline\Pi^-$, where $\sigma(M)$
denotes the spectrum of $M$.  The left spectral factor $\check S_{21}$ associated with
$\check P$ is called the {\em outer spectral factor}. Its inverse is analytic
in $\Pi^+$. 

\begin{proposition}\label{spectralfactors}
Let $S$ be  a  Schur function which is strictly contractive at
infinity. Let $P$ and $\widetilde P$ be two distinct solutions of (\ref{RE})
and 
\[{\cS}_P= \begin{block}{cc} S_{11} & S_{12} \\
S_{21} & S \end{block}, ~~~ {\cS}_{\widetilde P}= \begin{block}{cc}
\widetilde S_{11} & \widetilde S_{12} \\
\widetilde S_{21} & S \end{block},   \]
the  inner extensions of $S$  associated to them {\it via} 
Theorem \ref{thm:GR}.  \\
Then, the matrix $Q=S_{21}^{-1}\widetilde S_{21}$ is
well-defined and unitary on the imaginary axis. 
Its McMillan degree  coincides
with the rank of $\widetilde P-P$ and $Q$ is  inner if and only if 
$\widetilde P\geq P$.
\end{proposition}

\begin{proof}
Assuming that $S$ is strictly contractive at infinity, 
$S_{21}$ is invertible and  we may define $Q=S_{21}^{-1}\widetilde S_{21}$.
We have that
\begin{equation}
\label{factgen}
QQ^*=S_{21}^{-1}\widetilde S_{21}\widetilde S_{21}^*S_{21}^{-*}
=S_{21}^{-1}(I_p-SS^{*})S_{21}^{-*}
=S_{21}^{-1}(S_{21}S_{21}^{*}) S_{21}^{-*}=I_p,
\end{equation}
so that $Q$ is unitary. A realization of $Q$ can be computed from  the
realizations of $S_{21}$ and $\widetilde S_{21}$ of 
Theorem~\ref{thm:GR}, using (\ref{Rprod}) and (\ref{realinv}):
\[
 Q = \left(
\begin{array} {c|c} A-B_1D_{21}^{-1}C & B_1D_{21}^{-1}\\ \hline
                   -D_{21}^{-1}C & D_{21}^{-1}
\end{array}
\right)
\left(
\begin{array} {c|c} A & \widetilde B_1\\ \hline
                     C & D_{21}
\end{array}
\right)
=
\left(
\begin{array} {cc|c} A-B_1D_{21}^{-1}C & B_1D_{21}^{-1}C &  
                            B_1\\
                     0 & A & \widetilde B_1 \\ \hline
                   -D_{21}^{-1}C & D_{21}^{-1} C &
                   I_p
\end{array}
\right).
\]
Applying a change of variables using 
$ T = \begin{block}{cc} I_n & - I_n \\ 0 & I_n \end{block}$
we get 
\begin{equation}
\label{realQnm}
Q= 
\left(
\begin{array} {cc|c} A-B_1D_{21}^{-1}C & 0 &  
                            B_1-\widetilde B_1\\
                     0 & A & \widetilde B_1 \\ \hline
                   -D_{21}^{-1}C & 0 & I_p
\end{array}
\right)= 
\left(
\begin{array} {c|c} A-B_1D_{21}^{-1}C &  
                            B_1-\widetilde B_1 \\ \hline
                   -D_{21}^{-1}C & I_p
\end{array}
\right).
\end{equation}
From (\ref{b})  we draw
$ B_1-\widetilde B_1 =
(\widetilde P-P)C^*D_{21}^{-1}
$
and 
\begin{equation}\label{Qrealization}
 Q=
\left(
\begin{array} {c|c} Z & (\widetilde P-P)C^* D_{21}^{-1}\\ \hline
                 -  D_{21}^{-1} C & I_p
\end{array}
\right)
\end{equation}
where 
$Z$ is given by (\ref{Z}).

Set $\Gamma  = \widetilde P-P$.
Since $\Gamma$ is Hermitian, the singular value decomposition 
can be written 
\begin{equation}
\label{V}
\Gamma = V
\begin{block}{cc} 0 & 0 \\ 0 & \Gamma_0 \end{block} V^*, 
\end{equation}
where $V$ is unitary and $\Gamma_0$ is real and diagonal.
Let $d\times d$ be the size of $\Gamma_0$, and partition $V$ as $V=[ V_1 \; \, V_2]$
where $V_1$ is of size $n\times(n-d)$ and $V_2$ of size $n\times d$.
Note that the columns of $V_1$ span the kernel of $\Gamma$.
In another connection, it holds that 
\begin{eqnarray*}
\cR(\widetilde P)-\cR(P)&=&\Galpha \Gamma+\Gamma
\Galpha^*+\widetilde P\Ggamma \widetilde P-P\Ggamma P\\
&=&Z\Gamma -P\Ggamma\Gamma+\Gamma Z^*-\Gamma\Ggamma P+\Gamma\Ggamma
\widetilde P+P\Ggamma\Gamma\\
&=& Z\Gamma+\Gamma Z^* +\Gamma \Ggamma \Gamma,
\end{eqnarray*}
which is zero since both $P$ and $\widetilde P$ are solutions to the Riccati
equation (\ref{RE}). Thus $\Gamma$ is a solution to the Riccati equation
\begin{equation}\label{REgamma} 
Z\Gamma+\Gamma Z^*  +\Gamma \Ggamma \Gamma= 0.
\end{equation}
Partitioning  $V^*ZV$ into
\begin{equation}
\label{partiV}
V^*Z V = \begin{block}{cc} Z_{11} & Z_{12} \\ Z_{21} & Z_{22} \end{block}, 
\end{equation}
we can rewrite (\ref{REgamma}) as 
\[ 
\begin{block}{cc} Z_{11} & Z_{12} \\ Z_{21} & Z_{22} \end{block}
\begin{block}{cc} 0  & 0 \\ 0 & \Gamma_0  \end{block} 
+
\begin{block}{cc} 0  & 0 \\ 0 & \Gamma_0 \end{block} 
\begin{block}{cc} Z_{11} & Z_{12} \\ Z_{21} & Z_{22} \end{block}^*
+
\begin{block}{cc} 0  & 0 \\ 0 & \Gamma_0  \end{block} 
V^* \Ggamma V
\begin{block}{cc} 0  & 0 \\ 0 & \Gamma_0  \end{block}= 0,  
\]
that is  
\[ 
\begin{block}{cc} 0   & Z_{12}\Gamma_0\\ 
\Gamma_0 Z_{12}^* & Z_{22} \Gamma_0 + \Gamma_0 Z_{22}^* + \Gamma_0
V_2^*\Ggamma V_2 \Gamma_0 
\end{block} 
= 0.\]
Therefore 
\begin{equation}
\label{Z21null}
Z_{12} = 0
\end{equation}
 and 
\begin{equation}
\label{LyapLambda}
Z_{22} \Gamma_0 + \Gamma_0  Z_{22}^* + \Gamma_0 V_2^*\Ggamma V_2 \Gamma_0=0. 
\end{equation}
Using $V$ as change of coordinates in the state space, we obtain a new realization for  $Q$: 
\begin{equation}
\label{realQbt}
Q=
\left(
\begin{array} {c|c} V^*ZV & 
D_{21}^{-1} V^* \Gamma C^* 
\\ \hline                  
 CV D_{21}^{-1} & I_p
\end{array}\right)
=
\left(
\begin{array} {c|c}\begin{array}{cc}Z_{11} & 0\\Z_{21} & Z_{22} \end{array}
& 
\begin{array}{cc} 0 \\ \Gamma_0 V_2^* C^*D_{21}^{-1}  \end{array} 
\\ \hline                  
\begin{array}{cc} D_{21}^{-1} CV_1 & D_{21}^{-1} CV_2 \end{array} & I_p
\end{array}
\right)
\end{equation}
where we used (\ref{Z21null}). This readily reduces to
\begin{equation}\label{Qmin}
 Q= 
\left(
\begin{array} {c|c} Z_{22} & \Gamma_0 V_2^* C^* D_{21}^{-1}
\\ \hline                  
D_{21}^{-1}C V_2 & I_{p}
\end{array}
\right),
\end{equation}
and by (\ref{LyapLambda}) we also have (recall $\Gamma_0$ is invertible)
that $\Gamma_0^{-1}$ solves the Lyapunov equation
\begin{equation}
\label{LyapZ22}
Z_{22}^*\Gamma_0^{-1} + \Gamma_0^{-1}  Z_{22}+V_2^*\Ggamma V_2=0.
\end{equation}
From (\ref{Qmin}), it is clear that the degree of $Q$ is at most the size of $\Gamma_0$
 which is the rank of $\Gamma=\tilde P-P$. \emph{We claim} that the realization (\ref{Qmin}) is minimal.
 Indeed, it is observable because we started from the observable realization (\ref{realQnm})
 and we just restricted ourselves in step (\ref{realQbt})-(\ref{Qmin}) 
 to some invariant subspace of the dynamics matrix
 in the state space. To check reachability, we use Hautus's test that no nonzero left eigenvector
 $x^T$ of $Z_{22}$ associated, say, to some eigenvalue $\lambda$,
 can lie in the left kernel of $\Gamma_0 V_2^* C^* D_{21}^{-1}$; for then
 (\ref{LyapLambda}) and (\ref{def:Ggamma}) together imply that $x^T\Gamma_0$, which is nonzero 
 as $\Gamma_0$ has full rank, is a left eigenvector
 of $Z_{22}^*$ associated to $-\lambda$. Thus $\Gamma_0^*\bar{x}$ would be an eigenvector of $Z_{22}$
 and by construction it lies in the kernel of $D_{21}^{-1}C V_2$, contradicting the
 observability of (\ref{Qmin}). \emph{This proves the claim}, to the effect that
 the degree of $Q$ is in fact equal to the 
 rank of $\Gamma=\tilde P-P$. Now, from classical properties of solutions to Lyapunov equations
 \cite[th.6.5.2]{BGR},
 the number of poles of $Q$ in $\Pi^+$ is equal to
 the number of negative eigenvalues of $\Gamma_0^{-1}$ solving (\ref{LyapZ22}).
\end{proof}

\begin{remark}
\label{remSF}
A similar result holds true for the right inner factors: the matrix
$R=\widetilde S_{12} S_{12}^{-1}$ is unitary on the imaginary axis and inner
if and only if $\widetilde P\leq P$.
\end{remark} 
The outer spectral factor play an important role in what follows, due to
the fact that \emph{any} (not necessarily minimal) left  spectral factor, say, $\sigma$ of $I_p-SS^*$ 
can be factored as $\sigma=\check S_{21} Q$ where $Q$ is inner. Indeed, the strict contractivity of $S$ 
at infinity entails that $\sigma$ is invertible as a rational matrix, and then computation 
(\ref{factgen}) with $S_{21}$, $\widetilde S_{21}$ replaced by $\check S_{21}$, $\sigma$ shows that
the rational matrix $Q={\check S_{21}}^{-1}\sigma$ is unitary on the imaginary axis. 
In particular it cannot have a pole there, and since it is analytic in $\Pi^+$ as 
$\sigma$ is stable and 
$\check S_{21}$ outer, we conclude that $Q$ is stable thus inner, as desired.

\subsection{Inner extensions of higher degree}
We shall be interested in inner rational extensions where
we allow for an increase in the McMillan degree, and we will 
base our analysis on the next proposition. For the proof, we need a notion of coprimeness:
two inner-functions $L_1$, $L_2$ are \emph{right coprime} if one cannot write
$L_1=G_1 J$ and $L_2=G_2 J$ with $G_1$, $G_2$, $J$ some inner functions and $J$ non-constant.
It is well known \cite{F81,BO} that this is equivalent to require 
the existence of two stable rational matrix functions 
$X_1$ and $X_2$ such that the following Bezout equation holds: 
\begin{equation}\label{bezL}
X_1(s)L_1(s) +  X_2(s)L_2(s) = I_p. 
\end{equation}
Left coprimeness is defined in a symmetric way.

\begin{proposition}
\label{AllLosslessExtensions}
All rational inner extensions of a Schur function $S$, contractive at
$\infty$, 
 can be written on the form
\begin{equation}
\label{factextgen}
  \begin{block}{cc} L & 0 \\
 0 & I_p \end{block}
\cS_P
 \begin{block}{cc} R & 0 \\
 0 & I_p \end{block}
\end{equation}
where $L$, $R$ and $\cS_P$ 
are inner, 
and $\cS_P$ is an extension of $S$ at the same McMillan degree, obtained from
a solution $P$ of the Riccati equation (\ref{RE}).
\end{proposition}

\begin{proof}
Let
\[\check \cS=\begin{block}{cc} \check S_{11} & \check S_{12} \\
\check S_{21} & S \end{block}\]
be the  inner extension of $S$ of degree $n$ associated with the minimal
solution $\check P$ of (\ref{RE}), so that $\check S_{21}$ is  the outer left spectral factor of 
$I_p-S S^* $. 
For an arbitrary rational inner extension   
\[ \Sigma = \begin{block}{cc} \sigma_{11} & \sigma_{12} \\
\sigma_{21} & S \end{block}\]
of $S$, we mechanically obtain since $\Sigma\Sigma^{*}=\check \cS {\check \cS}^*=I_{2p}$ that
\[
\check S_{21}\check S_{21}^*  =  I_p-SS^*  = \sigma_{21}\sigma_{21}^* ~~~~{\rm and}~~~~
\check S_{12}^* \check S_{12}  =  I_p-S^* S  = \sigma_{12}^*\sigma_{12}.
\]
Therefore, computing as in  (\ref{factgen}), we can write
$\sigma_{21}=\check S_{21}R$ and $\sigma_{12}=L\check S_{12}$
where $R={\check S_{21}}^{-1}\sigma_{21}$ and $L=\sigma_{12}{\check S_{12}}^{-1}$ are rational 
matrices that are unitary on the imaginary axis. By the discussion after Remark \ref{remSF},
$R$ is inner. Next, using again that $\Sigma\Sigma^{*}=I_{2p}$, we get 
$\sigma_{11}=-\sigma_{12} S^* \sigma_{21}^{-*} =
- L \check S_{12} S^* \check S_{21}^{-*} R =  L \check S_{11} R$,
and therefore 
\[ \Sigma = 
\begin{block}{cc} L & 0 \\
 0 & I_p \end{block}
\begin{block}{cc} \check S_{11} & \check S_{12} \\
\check S_{21} & S \end{block}
 \begin{block}{cc} R & 0 \\
 0 & I_p \end{block}.
\]

If $L$ is inner, we have finished because (\ref{factextgen}) holds with $P=\check P$. 
Otherwise, being a unitary rational function, 
$L$ can be factored  as $L=L_{1}L_2^*$  where 
$L_1$ and $L_2$ are right coprime inner functions, so that (\ref{bezL}) holds for some
stable transfer functions $X_1$, $X_2$. In fact,
the existence of such a factorization follows from the so-called
Douglas-Shapiro-Shields factorization \cite{DSS} as carried over to matrix-valued strictly non cyclic 
functions in \cite{F81}, see {\it e.g.} \cite{BO} for a detailed discussion of the rational case
in discrete time that translates immediately to continuous time by
linear fractional transformation.

From (\ref{bezL}) we deduce that $L_2^*(s)=X_1(s)L(s) +  X_2(s)$, therefore
\begin{equation}
\label{L2*SRstable}
 \begin{block}{cc} L_2^* & 0 \\
 0 & I_p \end{block}
\begin{block}{cc} \check S_{11} & \check S_{12} \\
\check S_{21} & S \end{block}
 \begin{block}{cc} R & 0 \\
 0 & I_p \end{block}= \begin{block}{cc} X_1 & 0 \\
 0 & I_p \end{block}\Sigma +\begin{block}{cc} X_2 & 0 \\
 0 & I_p \end{block}{\check \cS}  \begin{block}{cc} R & 0 \\
 0 & I_p \end{block}
\end{equation}
is stable. In particular, if we put ${S}_{12}:=L_2^* \check S_{12}$, we see that
\begin{equation}
\label{colint}
\begin{block}{c} S_{12} \\
 S \end{block}=\begin{block}{cc} L_2^* & 0 \\
 0 & I_p \end{block}
\begin{block}{c} \check S_{12} \\
 S \end{block}
\end{equation}
is stable.
Now, multiplication by a rational function analytic in $\overline{\Pi}^-$ 
(including at infinity) followed by the projection onto stable rational functions 
(obtained by partial fraction extension) \emph{cannot increase the McMillan degree};
this follows at once from Fuhrmann's realization theory \cite{F81,BO}. Therefore, we deduce from
(\ref{colint}) that the degree of $[S_{12}^T~S^T]^T$ is at most the degree of 
$[{\check S_{12}}^T~S^T]^T$. But the latter is equal to $n$ for Corollary \ref{innerext} as applied to
${\check \cS}^T$ implies that $\check S_{12}$ is a minimal right spectral factor of $I_p-S^* S$.
Hence the degree of $[S_{12}^T~S^T]^T$ is $n$ and
$S_{12}$ is again a minimal right spectral factor of $I_p-S^* S$.
Thus by (the transposed version of) Corollary \ref{innerext}, there exists an  inner
extension of $S$ of  McMillan degree $n$ associated with
$ S_{12}$: 
\[\cS_{P}=\begin{block}{cc} S_{11} & S_{12} \\
 S_{21} & S \end{block}.\]

As seen after Remark \ref{remSF}, one has $S_{21}=\check S_{21}R_{2}$ for some inner 
matrix $R_2$. Moreover, by the inner character of $S_P$ and $S_{\check P}$ and in view
of (\ref{colint}), we get 
\[S_{11}=-S_{12}S^*S_{21}^{-*}=-L_2^*{\check S }_{12}S^*{\check S}_{21}^{-*}R_2^{-*}=
L_2^*{\check S }_{11}R_2^{-*}.\]
Altogether, this implies
\begin{equation}
\label{SPR2=L2S} 
\begin{block}{cc}  S_{11} &  S_{12} \\
 S_{21} & S \end{block}
\begin{block}{cc} R_2^* & 0 \\
 0 & I_p \end{block}
=
\begin{block}{cc} L_2^* & 0 \\
 0 & I_p \end{block}\begin{block}{cc} \check S_{11} & \check S_{12} \\
\check S_{21} & S \end{block}.
\end{equation}

Next,  \emph{we contend} that the inner functions ${\cal S}_P$ and
${\rm diag}\,[R_2,I_p]$ 
are right coprime. Indeed, if $R_c$
is a right common inner factor, we get in particular 
$R_c={\rm diag}\,[R_2, I_p]J^*$  for some inner $J$, so the last $p$ rows of both $J$ and $J^*$
are analytic in $\bar{\Pi}^+$, {\it i.e.} they are constant. This entails 
\begin{equation}
\label{factd}
R_{c}=U\begin{block}{cc} R_3 & 0 \\
 0 & I_p \end{block}\end{equation}
 for some constant unitary matrix $U$.
But then $\cS_{ P}R_{c}^{*}U$ is inner since $R_c$ divides $\cS_{P}$, and it is
an extension of $S$ by (\ref{factd}). Taking determinants, we see that this extension has
McMillan degree equal to ${\rm deg}\cS_P-{\rm deg}R_3$. As this degree is at least equal to 
that of $S$ which is also that of $\cS_P$, we conclude that $R_3$ thus also $R_c$ are constant, 
\emph{which proves our contention}.
 
By the coprimeness above, there exist stable rational matrices  
$X$ and $Y$ such that
\begin{equation}
\label{Bezoutcalin}
X\,\begin{block}{cc}  S_{11} &  S_{12} \\
 S_{21} & S \end{block}\,+\,Y\,\begin{block}{cc} R_2 & 0 \\
 0 & I_p \end{block}=I_{2p}.
\end{equation}
From (\ref{L2*SRstable})-(\ref{SPR2=L2S}), we see that the
 product $\cS_{ P}\,{\rm diag}[ R_2^{*} R ,I_p] $ is stable. 
Therefore, right multiplying  (\ref{Bezoutcalin}) by ${\rm diag}[ R_2^{*} R ,I_p] $,
we deduce that the latter is also stable hence $R_1:=R_{2}^* R$ is inner. Finally,
\[\Sigma=\begin{block}{cc} L_1 & 0 \\
 0 & I_p \end{block} \cS_P \begin{block}{cc} R_1 & 0 \\ 
 0 & I_p \end{block}\]
is indeed of the form (\ref{factextgen}), as wanted.\end{proof}

\section{ Symmetric unitary extensions}\label{symun}

We assume from now on that the Schur function $S$ is symmetric
i.e. $S=S^\transpose$ and we consider  a  symmetric 
realization,
\begin{equation}\label{realsym:S}
 S = \left( \begin{array}{c|c} A & B \\ \hline C & D 
\end{array}\right),~~~A=A^\transpose,~~~ B=C^\transpose,~~~ D=D^\transpose
\end{equation}
Such a realization always exists thanks to Theorem 5 in \cite{FH95}.

It follows that the matrices $\Galpha, \Gbeta, \Ggamma$ 
defined by (\ref{def:Galpha}), (\ref{def:Gbeta}) and (\ref{def:Ggamma})
satisfy
\begin{equation}
\label{cA:symmetry}
\left\{
\begin{array}{ccc}
\Galpha&=&\Galpha^\transpose\\
\Gbeta&=&(\Ggamma)^\transpose.
\end{array}\right.
\end{equation}

If the extension $S_P$ associated to some solution $P$ of (\ref{RE}) {\it via} Theorem \ref{thm:GR} is
symmetric, then the matrix $Z$ defined in (\ref{Z}) must be similar to $-Z^*$ 
since they are, by the computations (\ref{dynZ}) and (\ref{dynZ*}), dynamics matrices of
$S_{21}^{-1}$ and $S_{12}^{-1}$ respectively.  Therefore,
in view of (\ref{chiAsplit}), the characteristic polynomial $\chi_\nocH$ must be
of the form $\pi(s)^2$. As this may not be the case, 
a symmetric inner extension of $S$ preserving the McMillan degree may well fail to exist.  However,
as we shall see ({\it cf.} also \cite{AV}), symmetric inner extensions of higher degree do exist. 
We will first give a simple method to construct (possibly unstable)  symmetric unitary extensions 
of $S$ using the inner extensions $S_P$ provided to us by Theorem \ref{thm:GR}.

As the Riccati equation (\ref{RE}) does admit a Hermitian solution and since
the pair $(\Galpha,\hC)$ in (\ref{def:Galpha})-(\ref{def:Ggamma}) is observable,
as follows immediately from Hautus's test on using the observability of $(A,C)$,
the partial multiplicities ({\it i.e.} the sizes of the Jordan blocks)
of the pure imaginary eigenvalues of
$\nocH$ (if any) are all even, see  \cite[th.2.6]{LR1} or \cite[th. 7.3.1]{LR2}. 
Let $2 n_0$ be the dimension of the
spectral subspace of $\nocH$ corresponding to all its pure imaginary eigenvalues.
Using (\ref{simcalA}) and grouping together the roots of even multiplicity,
we can write the characteristic polynomial of $\nocH$ in the form 
\begin{equation}
\label{chiAfact}
\chi_\nocH(s)=\pi(s)^2\chi_\kappa^+(s)\chi_\kappa^-(s),
\end{equation}
where $\pi(s)$, $\chi_\kappa^+(s)$, and $\chi_\kappa^-(s)$ are polynomials in $s$, 
and where $\chi_\kappa^+$ has $\kappa$ simple roots in $\Pi^+$ while 
$\chi_\kappa^-=(\chi_\kappa^+)^*$. 
Then $\kappa$ is the number of distinct eigenvalues of $\nocH$ in $\Pi^+$ with
odd multiplicity, and $\pi(s)$ has degree greater than or equal to $n_0$ by what precedes.

\begin{proposition}\label{extension}
Assume $S$ is  a symmetric Schur function strictly contractive at
infinity, and let (\ref{realsym:S}) be a symmetric realization. For $P$ a Hermitian solution  
of (\ref{RE}), let further
\[{\cS}_P= \begin{block}{cc} S_{11} & S_{12} \\
S_{21} & S \end{block},\]
be the inner extension of the same degree associated by Theorem \ref{thm:GR}. \\
Then, the matrix $Q=S_{21}^{-1}S_{12}^\transpose$ is
well-defined, unitary on the imaginary axis, and 
\[ \Sigma_P=\begin{block}{cc} S_{11} & S_{12} \\
S_{21} & S \end{block}
 \begin{block}{cc} Q & 0 \\
 0 & I_p \end{block}
=
\begin{block}{cc} S_{11}S_{21}^{-1}S_{12}^\transpose & S_{12} \\
S_{12}^\transpose & S \end{block},
\]
is a symmetric extension of $S$ which is unitary on the imaginary axis. 
The McMillan degree of $Q$ coincides
with the rank of $P^{-\transpose}-P$, and $Q$ is  inner if and only if 
$P^{-\transpose}-P$ is positive semi-definite. In this case, 
$\Sigma_P$ is inner and its McMillan degree  is ${\rm deg} S+{\rm deg} Q$. 
\\
Moreover, ${\rm deg} \,Q$ is in any case greater than or equal to  $\kappa$, the number of distinct
eigenvalues of $\nocH$ in $\Pi^+$ with 
odd multiplicity. 
\end{proposition}

\begin{proof}
From (\ref{cA:symmetry}) we see that 
 $\cR(P)=0$  if and only if  $\cR(P^{-\transpose})=0$, and that
 $\cS_{P^{-\transpose}}= \cS_P^\transpose$. Therefore, we may appeal to  
Proposition \ref{spectralfactors} with $\widetilde P=P^{-\transpose}$
and $\widetilde S_{21}=S_{12}^T$.
Thus,  $Q=S_{21}^{-1}S_{12}^\transpose$ is unitary on the imaginary axis, its degree $d$ is equal to the rank of $P^{-T}-P$ and $Q$ is inner if
 and only if $P^{-T}\geq P$. 
The computations in the proof of Proposition \ref{spectralfactors} apply,
so if (\ref{V}) is the singular value decomposition of
$\Gamma=P^{-T}-P$ and $Z$ is defined by (\ref{Z}), we get from (\ref{partiV}) and (\ref{Z21null}) that 
 \begin{equation}
\label{semblableZ}
V^*ZV=\left[\begin{array}{cc}V_1^*Z V_1
    & V_1^*Z V_2 \\ V_2^*Z V_1& V_2^*Z
    V_1\end{array}\right]=\left[\begin{array}{cc}Z_{11} 
    & 0 \\Z_{21} & Z_{22}\end{array}\right].
\end{equation}
On the other hand, if we set ${\widetilde Z}:=-P Z^* P^{-1}$, it follows from (\ref{RE}) and (\ref{Z})
that
\begin{equation}
\label{unerel}
{\widetilde Z}= -P Z^* P^{-1}=-P(\Galpha^*+\Ggamma P)^* P^{-1}=(\Galpha P+\Gbeta)^*
P^{-1}=\Galpha +\Gbeta P^{-1}.
\end{equation}
But since the columns of $V_1$ span the kernel of $\Gamma$,
we get $V_1^*(P-P^{-T})=0$ that transposes into $P^T \overline V_1=P^{-1}\overline V_1$, hence in
view of (\ref{unerel}) and (\ref{cA:symmetry})
\[{\widetilde Z} \overline V_1= (\Galpha +\Gbeta P^{T})\overline V_1=Z^T
\overline V_1.\]
As $V$ is unitary, we have thus arrived at a similarity relation of the form:
\begin{equation}
\label{semblableZtilde}
V^T{\widetilde Z}\overline V= \left[\begin{array}{cc}V_1^T{\widetilde Z} \overline V_1
    & V_1^T{\widetilde Z} \overline V_2 \\ V_2^T{\widetilde Z}\overline V_1&
    V_2^T{\widetilde Z} \overline
    V_2\end{array}\right]\\
=\left[\begin{array}{cc}Z_{11}^T 
    & \widetilde Z_{12}  \\0 & \widetilde Z_{22} \end{array}\right].
\end{equation}
Now, since $\widetilde Z$ is similar to $-Z^*$, we see from (\ref{semblableZ}), (\ref{semblableZtilde}),
and (\ref{chiAsplit}) that $\chi_{Z_{11}}^2(s)$ must divide
$\chi_\nocH(s)$. Consequently the number of rows (or columns) of $Z_{11}$, which is the dimension
of the kernel of $P^{-T}-P$, cannot exceed the degree of $\pi(s)$ 
in (\ref{chiAfact}), namely $n-\kappa$. Therefore ${\rm rank}
(P^{-T}-P)\geq \kappa$, as announced.
\end{proof}

\begin{remark}
Note that ${\rm deg}Q=0$, i.e. $Q$ is constant, 
if and only if $P^{-\transpose}=P$.
\end{remark}

\section{Minimal symmetric inner extensions}\label{minsym}
\label{SecMinSymCom}

\begin{lemma}
\label{PminPmax}
 Let $\check P$ and $\hat P$ be the minimal and the maximal solution to  the
 Riccati equation (\ref{RE}) associated with a symmetric Schur
function strictly contractive at infinity. Then
\[\check P^{-T}=\hat P.\]
\end{lemma}

\begin{proof}
Since $\check P$ is the minimal solution, for each
solution $P$ we have that $P\geq \check P$.
By symmetry, $P$ is a solution if and only if $P^{-T}$ is  a solution and moreover
$\check P^{-T}\geq P^{-T}$,
so that $\check P^{-T}$ must be the maximal solution.
\end{proof}

A symmetric inner extension of the symmetric Schur function $S$ may now be
obtained as follows. 

\begin{proposition}
\label{maxdegsymext}
Let 
\[S_{\check P} =   \left[ \begin{array}{cc}
    {\check S}_{11} &{\check S}_{12}\\ 
    {\check S}_{21} & S \\ 
  \end{array}\right] \]
be the extension associated with the minimal solution $\check P$ to
the Riccati equation (\ref{RE}). The  symmetric extension $\Sigma_{ \check P}$  given by 
 Proposition \ref{extension}:
  \[ \Sigma_{ \check P}=  \left[ \begin{array}{cc}{\check S}_{11} &{\check S}_{12}\\ 
    {\check S}_{21} & S \\ 
  \end{array}\right] 
   \left[ \begin{array}{cc}{\check Q} &0\\ 
   0 & I_p \\ 
  \end{array}\right],   ~~~~{\rm with}~{\check Q}={\check S}_{21}^{-1}{\check S}_{12}^T,\]
is inner and has degree $2n-n_0$.
\end{proposition}

\begin{proof}
Since $\check P$ is the minimal solution of the Riccati equation, we have
that $\check P^{-T}-\check P\geq 0$, hence
by Proposition \ref{extension} the extension $\Sigma_{\check P}$  is
inner and  has degree $n+d$ where $d$ is the rank of
$\check P^{-T}-\check P$. By Lemma \ref{PminPmax}
$\check P^{-T}= \hat P$, and if we let  ${\rm ker}(\hat P-\check P)=\cN$ it is a classical fact
(see {\it e.g.} \cite[th.2.12]{LR1} or \cite{LR2}) that  $\cN$ is, for any
solution $P$ to (\ref{RE}), the spectral subspace of 
$Z=\Galpha+ P \Ggamma$ corresponding to all of its pure imaginary eigenvalues, and that
it  has dimension $n_0$. Therefore ${\rm deg} \,\Sigma_{\check P}=2n-n_0$, as desired.
\end{proof}

Finally, we shall construct from $\Sigma_{\check P}$ a symmetric inner extension of $S$ of minimal degree.
For this, we first establish a factorization
property which strengthens the
Potapov factorization (\ref{Potapov}) in the symmetric case.

For $\rm Re \;\xi>0$, recall from (\ref{defbxi}) and (\ref{defBxi})
the definitions of $b_\xi$ and $B_\xi$.
Let $U$ be a unitary matrix whose first column is a unit vector $u\in\CC^p$,
then 
\begin{equation}\label{Bxiu}
B_{\xi,u}=U B_\xi(s)U^*=I_p+( b_\xi(s)-1)uu^*
\end{equation}
is an elementary Blaschke factor with
the following properties (compare \cite[chap.1]{Dym}):
\begin{eqnarray}
\label{invQxiu}
B_{\xi,u}^*(s)&=&I_p+( b_\xi(s)^{-1}-1)uu^*\\
\label{detQxiu}
\det B_{\xi,u}(s)&=&b_\xi(s).
\end{eqnarray}
Note that any elementary Blaschke factor can be written in the form 
$B_{\xi,u}V$ for some unit vector $u$ and some unitary matrix  $V$  .

\begin{lemma}\label{comporthdec}
For $F$ a $p\times p$ symmetric complex matrix, there exists a unitary matrix $U$ such that 
\begin{equation}
\label{symmsvd}
F=U\Lambda U^{T}
\end{equation}
with $\Lambda$ a non-negative diagonal matrix.If $F$ has rank $k$, the $p-k$ first columns of $U$ 
may be construed to be any orthonormal basis of the kernel of $F$.
\end{lemma}
\begin{proof} The existence of (\ref{symmsvd}) is proved in \cite[cor.4.4.4]{HJ} under the name of
Takagi's factorization. Permuting the columns of $U$, we may arrange the diagonal entries of
$\Lambda$ in non-decreasing order and then the first $p-k$ columns form an orthogonal basis of 
${\rm Ker}F$. Right multiplying $U$ by a block diagonal unitary matrix of the form
${\rm diag}\,[V,I_k]$ we may clearly trade that basis for any other.
\end{proof}

\begin{lemma}\label{IntCondLemma}
Let $T(s)$ be a symmetric inner function and $\xi\in\Pi^+$
a  zero of $T(s)$. Suppose that there exists 
a  unit vector  $u$ such that
the interpolation conditions
\begin{eqnarray}\label{IntCond}
T(\xi)u&=&0\\
u^T T'(\xi) u&=&0
\end{eqnarray}
are satisfied. 
Defining $B_{\xi,u}(s)$ as in (\ref{Bxiu}), then, 
\[R(s)= B_{\xi,u}(s)^{-T }T(s)B_{\xi,u}(s)^{-1},\]
is analytic at $\xi$ and thus a symmetric inner function of degree $N-2$.
\end{lemma}

\begin{proof}
We give a proof of this result 
which follows Potapov's approach to 
the multiplicative structure of $J$-inner functions \cite{Pot60}.
For simplicity, we use Landau's notation $O(s)^k$ for the class of (scalar or matrix-valued) functions $f(s)$
such that $\|f(s)\|/|s|^k$ is bounded for $|s|$ sufficiently small.
We also put $\|f(s)\|$ for the operator norm of $f(s)$.
Write the Taylor expansion of
$T(s)$ about $\xi$ as
\[T(s)=T(\xi)+(s-\xi)T^\prime(\xi)+O(s-\xi)^2,\]
and  apply Lemma \ref{comporthdec} to obtain
Takagi's factorization of $T(\xi)$ in the form
\begin{equation}
\label{TakagiT}
U^T\,T(\xi)\,U={\rm diag}(0,\ldots,0,\rho_1,\ldots,\rho_r),
\end{equation}
for some unitary matrix $U$ whose first column is $u$.
Next, define a matrix-valued function $T_1$ by
\begin{eqnarray*}
T_1(s)&:=&T(s)B_{\xi,u}(s)^{-1}\\
&=& T(\xi)B_{\xi,u}(s)^{-1}+(s-\xi)T^\prime(\xi)B_{\xi,u}(s)^{-1}+O(s-\xi)\\
&=&(U^*)^T{\rm
  diag}(0,\ldots,0,\rho_1,\ldots,\rho_r)\,B_{\xi}(s)^{-1}U^*+(s-\xi)T^\prime(\xi)B_{\xi,u}(s)^{-1}+O(s-\xi)\\
&=&T(\xi)+(s-\xi)T^\prime(\xi)B_{\xi,u}(s)^{-1}+O(s-\xi)
\end{eqnarray*}
where we have used (\ref{Bxiu}).
Clearly, $T_1(s)$ is analytic about $\xi$ and, since $B_{\xi,u}(s)$ has degree $1$, $
T_1(s)$ is  an  inner
function  of McMillan degree $N-1$. 
Now, using  (\ref{invQxiu}) and the symmetry of  $T(\xi)$ which implies $u^TT(\xi)=0$,
we  get the following interpolation condition for $T_1(s)$:
\[u^TT_1(\xi)=u^TT(\xi)+2{\rm Re} \,\xi u^TT^\prime(\xi)uu^*=0.\]
Applying what precedes to $T_1(\xi)^T$, we see that 
$R(s):=T_1(\xi)^TB_{\xi,u}(s)^{-1}$ is analytic at $\xi$ and thus an inner
function of degree $N-2$. Finally, we can write
\[R(s)= B_{\xi,u}(s)^{-T }T(s)B_{\xi,u}(s)^{-1},\]
 which achieves the proof.
\end{proof}

\begin{proposition}\label{recsteplemma}
Let $T(s)$ be a symmetric inner function of degree $N$ and suppose that
$T(s)$ has a  zero $\xi\in\Pi^+$  of {\em multiplicity
greater than} $1$. Then there exists an elementary Blaschke factor of the form 
$B(s)=B_{\xi,u}$, where  $u\in\CC^p$ is a unit vector in the kernel of $T(\xi)$, such that 
\[R(s)=B(s)^{-T}\, T(s)\, B(s)^{-1}\]
 is analytic at $\xi$. Thus $R(s)$ is inner, symmetric, and it has McMillan
degree $N-2$. Moreover, if $\cal V$ is a 2-dimensional subspace of the kernel of $T(\xi)$, we may impose
in addition that $u\in{\cal V}$.
\end{proposition}

\begin{proof}
Write the local Smith form of $T$ at $\xi$, namely (\ref{localSmithform}) 
where $S$ is replaced by $T$ and $k$ is equal to $p$.
Assume first that the kernel of $T(\xi)$ has dimension 1 over $\CC$.
Since $T$ is analytic at $\xi$ and 
$\xi$ is a zero of multiplicity at least $2$, the partial multiplicities must satisfy
$\nu_j=0$ for $1\leq j<p$ and $\nu_p\geq2$. Therefore,
if $e_p$ is the last element of the canonical basis of $\CC^p$, the vector
$u_0=F^{-1}(\xi)e_p$ spans the kernel of $T(\xi)$ and the
$\CC^p$-valued function $\phi(s)=F^{-1}(s)e_p$ is such that $T(s)\,\phi(s)$ has a zero of 
order at least $2$ at $\xi$:
\begin{equation}
\label{Tphi}
T(s)\,\phi(s)=O(s-\xi)^2.
\end{equation}
Let us write the Taylor series of
$T(s)$ and $\phi(s)$ about $\xi$:
\begin{eqnarray*}
T(s)&=&T(\xi)+(s-\xi)T^\prime(\xi)+O(s-\xi)^2,\\
\phi(s)&=&u_0+ u_1(s-\xi)+O(s-\xi)^2,
\end{eqnarray*}
for some $u_1\in\CC^p$. Then
\[T(s)\phi(s)=T(\xi)u_0+(s-\xi)(T^\prime(\xi)u_0+T(\xi)u_1)+O(s-\xi)^2,\]
and we must have
\begin{eqnarray*}
T(\xi)u_0&=&0\\
T^\prime(\xi)u_0+T(\xi)u_1&=&0.
\end{eqnarray*}
By symmetry the first equation implies $u_0^TT(\xi)=0$ and  then the second one
yields 
$u_0^TT'(\xi)u_0=0$. Since $u_0\neq0$, we obtain the desired result
from Lemma \ref{IntCondLemma} with $u=u_0/\|u_0\|$.

Assume next that the kernel of $T(\xi)$ has dimension at least 2 over $\CC$
and let $\cal V$ be a 2-dimensional subspace. Two cases can occur:

(i) There exists a non-zero $v\in{\cal V}$ such that $T(s)v$ vanishes at $\xi$ with order 
at least 2.
On grouping the partial multiplicities in such a way that $\sigma_j=0$ for $1\leq j\leq j_0$,
$\sigma_j=1$ for $j_0<j\leq j_1$, and $\sigma_j\geq2$ for $j_1<j\leq p$, we deduce that
in the decomposition
\[F(s)v=\sum_{j=1}^{j_0}\phi_j(s)e_j+\sum_{j=j_0+1}^{j_1}\phi_j(s)e_j+
\sum_{j=j_1+1}^{p}\phi_j(s)e_j\]
the functions $\phi_j$ vanish at $\xi$ with order at least 2 for $1\leq j\leq j_1$ and 
at least 1 for $j_0<j\leq j_1$. In particular $F^\prime(\xi)v$ lies 
the span of the $e_j$ for $j_0<j\leq p$ which is the kernel of $T(\xi)F^{-1}(\xi)$.
Therefore if we set $\phi(s):=F^{-1}(s)F(\xi)v$, we get an analytic function
about $\xi$ such that $T(\xi)\phi(\xi)v=0$ and 
\[(T\phi)^\prime(\xi)= T^\prime(\xi)v-T(\xi)F^{-1}(\xi)F^\prime(\xi)v=0+0=0\]
so that (\ref{Tphi}) holds. The result now follows as in the previous part of the proof with $u_0=v$.

(ii) Each non-zero $v\in{\cal V}$ is such that $T(s)v$ vanishes at $\xi$ with order 1.
By Takagi's factorization
\[U^T T(\xi)U={\rm diag}(0,\ldots 0,\rho_1,\rho_2,\ldots \rho_r),\] 
for some unitary matrix $U$, where $0\leq r\leq p-2$ and $0<\rho_1\leq \rho_2\leq\ldots \rho_r$.
We may assume that the first two columns of $U$ form a basis of ${\cal V}$. Define a 
matrix-valued analytic function about $\xi$ by
\[Q(s):=U^T T(s)U={\rm diag}(0,\ldots ,0,\rho_1,\rho_2,\ldots, \rho_r)+(s-\xi)Q'(\xi)+O(s-\xi)^2.\]
The matrix $Q'(\xi)$ is symmetric, and its $2\times 2$ left
upper-block $R_1(\xi)$  is invertible. Otherwise indeed, there would exist a non-zero 
$y=[y_1,y_2,0,\ldots,0]^T\in\CC^p$ such that $T(s)Uy$ vanishes at $\xi$ with order 2, but
this contradicts our standing assumption since $Uy\in{\cal V}$ by the choice of $U$.

We  now exhibit  a non-zero vector $w$ which satisfies the interpolation
conditions
\begin{eqnarray}
\label{intcond0}
Q(\xi)w&=&0\\
\label{intcond2}
w^TQ'(\xi)w&=&0.
\end{eqnarray}
To meet (\ref{intcond0}), we put $w=(x,0,\ldots,0)^T$ with $x\in\CC^2$ so that 
$w^TQ'(\xi)w=x^TR_1(\xi)x$. By Lemma \ref{comporthdec} we may write 
$U_1^TR_{1}(\xi)U_1=\Lambda$, with $U_1$ is a $2\times2$ unitary matrix, where
$\Lambda=\diag[\lambda_1^2, \lambda_2^2]$ is strictly positive.
Define now $v^T:=  \frac{1}{\sqrt{\lambda_1^2+\lambda_2^2}} \left[ \begin{array}{cc}
    \lambda_2 & i\lambda_1 
  \end{array} \right]$. Obviously
\[v^T\Lambda v=\frac{1}{\lambda_1^2+\lambda_2^2} [\lambda_2, -i \lambda_1]
\left[ \begin{array}{cc} 
    \lambda_1^2 & 0 \\ 
  0  & \lambda_2^2 \\ 
  \end{array} \right]  \left[ \begin{array}{c}
    \lambda_2 \\ 
    -i \lambda_1 \\ 
  \end{array} \right]=0,
\]
hence setting $x=U_1v$ we have that $w^TQ'(\xi)w=x^TR_1(\xi)x=v^T\Lambda v=0$,
as desired. Finally, if we let  $u=U[w^T,0]^T$, we get $u\in{\cal V}$ and putting
$B(s):=I_p+(b_{\xi}(s)-1)uu^*$ 
we see from Lemma \ref{IntCondLemma} that $B^*(s)^TT(s)B^*(s)$ is analytic about
$\xi$, which achieves the proof.
\end{proof}

\begin{theorem}\label{mindegrep}
Let  $S$  be  a $p\times p$ symmetric Schur function of McMillan degree $n$
which is strictly contractive at infinity.
Let further $\nocH$ be the state characteristic matrix defined by (\ref{def:A}) for some
minimal symmetric realisation $(A,B,C,D)$ of $S$. 
Assume that $\chi_\nocH$ has exactly $\kappa$ distinct 
  roots in $\Pi^+$ with odd algebraic multiplicity.
 Then, $S$ has a $2p\times 2p$ symmetric
  inner extension of McMillan degree $n+\kappa$, and this extension has
  minimal degree among all such extensions of $S$.
\end{theorem}

\begin{proof}
 We apply Proposition
\ref{recsteplemma} to the symmetric inner extension of Proposition 
\ref{maxdegsymext}:
\[\Sigma_{\check P}=\cS_{\check P} 
 \left[ \begin{array}{cc}{\check Q} &0\\ 
   0 & I_p \\ 
  \end{array}\right]
\]
in which $\check P$ is the minimal solution to (\ref{RE}).
Recall  that $\check P$ is the only solution to the Riccati equation such that 
$\sigma(\Galpha+\check P\Ggamma)\subset {\bar \Pi}^-$ (see section
\ref{relspec}), or equivalently  
$\sigma(\Galpha^*+\Ggamma \check P)\subset {\bar \Pi}^+$. By (\ref{simcalA}), (\ref{dynZ}), and Proposition 
\ref{maxdegsymext}, the eigenvalues 
of $H$ in $\Pi^-$, which are (counting multiplicity)
$n-n_0$ in number by (\ref{chiAfact}), are precisely the poles of $\check Q$.
Likewise the eigenvalues of $H$ in $\Pi^+$ are the zeros of
$\check Q$, that are reflected from its poles across the imaginary axis with corresponding
multiplicities. 

Let $\xi_1\in\Pi^+$ be one of these  with multiplicity strictly greater than one.  
By proposition \ref{recsteplemma}, there exists an inner function  $B=B_{\xi_1,u}$  such that
$T_1:=B^{-T}\Sigma_{\check P}
B^{-1}$ is inner of degree $2n-n_0-2$. 
Clearly $T_{1}$ is symmetric and moreover, since 
\[{\rm Ker}  \left[ \begin{array}{cc}
    {\check Q}(\xi_1) & ¥ \\ 
    ¥ & I_{p} \\ 
  \end{array} \right]\subset {\rm Ker} \Sigma_{\check P}(\xi_1),\]
 it follows from the proposition 
that $u$ may be chosen of the form $[\widetilde{u}^T,0^T]^T$ (each block being of size $p$) 
in which case $B$ is of the form 
\[\left[ \begin{array}{cc}B_0 &0\\ 
   0 & I_p \\ 
  \end{array}\right],\]
so that $T_1$ is an extension of $S$.
The matrix $\nocH$ has  $n-n_0=\kappa+2\ell$ eigenvalues in $\Pi^+$, counting multiplicities.
Thus, we can perform  $\ell$ iterations to obtain an extension of degree 
\[2n-n_0-2\ell=n+\kappa.\]

We now prove that this extension has minimal McMillan degree.
Let $\Sigma$ be any symmetric extension of $S$. By Proposition
\ref{AllLosslessExtensions}, it  can be written in the form
\[ \Sigma = \begin{block}{cc} L & 0 \\
 0 & I_p \end{block}
\cS_P
 \begin{block}{cc} R & 0 \\
 0 & I_p \end{block}~~~~{\rm with}~~~~\cS_P=\begin{block}{cc} S_{11} & S_{12}\\
S_{21} & S_{22} \end{block},
\]
where $L$, $R$, are inner, 
and $\cS_P$ is an inner extension of $S$ of the same degree.
The extension $\Sigma$ being symmetric, we must have
\[
(L S_{12})^\transpose=S_{21}R \Leftrightarrow S_{21}^{-1}S_{12}^\transpose=R
\bar L.\]
By Proposition \ref{extension}, the degree of the 
unitary matrix $S_{21}^{-1}S_{12}^\transpose$ cannot be less than
$\kappa$.  
This yields
\[\kappa\leq \deg{S_{21}^{-1}S_{12}^\transpose=\deg R
\bar L}\leq \deg R + \deg L,\]
so that
\[n+\kappa\leq n+\deg R + \deg L=\deg \Sigma.\]
\end{proof}

\begin{corollary}\label{prop:main}
Let $S$ be a symmetric Schur function  of size $p\times p$
which is strictly contractive at infinity. Then, the following propositions 
are equivalent
\begin{itemize}
\item[(i)] $S$ has a symmetric inner extension of size $(2p)\times (2p)$
of the same McMillan degree,  
\item[(ii)] there is a Hermitian solution $P$ to the algebraic Riccati
  equation (\ref{RE}) that satisfies $PP^\transpose=I_p$, 
\item[(iii)]  the characteristic polynomial of the matrix $\nocH$ given by
  (\ref{def:Galpha})-(\ref{def:A}) can be written 
\begin{equation}  
\chi_\nocH(s) = \pi(s)^2
\end{equation}
for some polynomial $\pi$.
\end{itemize}
Moreover, all extensions (i) are parameterized
by (\ref{Wext}), where $U_1=U_2^\transpose$ and $P$ satisfies (ii).
\end{corollary}

\section{Completion of a real Schur function}\label{realcase}

We now assume that the Schur function $S(s)$ is real that is to say 
\[S(s)=\overline{S(\bar s)}.\]
It thus admit a realization (\ref{real:S}) in which $A$, $B$, $C$ and $D$ are
real matrices. We have that

\begin{theorem}
Let $P$ be a Hermitian solution to (\ref{RE}). Then, the associated
minimal unitary extension $S_P$ of $S$ is real if and only if $P$ is real.
\end{theorem}

\begin{proof} Formula  (\ref{Ireal}) for $S_P$ shows that if $P$ is real, then $S_P$ is
real.  Conversely, assume that $S_P$ is real. Formula (\ref{Ireal}) shows
that 
\[C(s I_n-A)^{-1} PC^*\]
is real for real $s$. This implies that $CA^kPC^*$ is real for
$k=0,1,\ldots,$ and from the observability of $(C,A)$ it follows that $PC^*$
is real. The equation  (\ref{RE}) can be rewritten in the form
\[PC^*(I_p-DD^*)^{-1}CP+(A+BD^*(I_p-DD^*)^{-1}C)P \]
\[+P (A^*+C^*(I_p-DD^*)^{-1}DB^*)+B(I_p-D^*D)^{-1}B^*=0,\]
and we see now that $AP+PA^*$ is real. Write $P=P_1+iP_2$ where $P_1$ and
$P_2$ are real matrices.  Since $CP$ is real, we have $CP_2=0$ and 
$P_2A^*+AP_2=0$. Premultiply the later equation by $C$ to obtain 
$CAP_2=0$, and by induction $CA^kP_2=0$, $k=0,1,\ldots,$.
Now, by the observability of $(C,A)$, $P_2=0$ and $P$ is real. 
\end{proof}

A real minimal realization of a symmetric function may fail to be symmetric. 
However there exist a minimal real realization which is signature symmetric (see \cite{YT} or
\cite[th.6.1]{F83}), that is to say
\[A^T=JAJ,~~~~B^T=CJ,~~~~C^T=JB,~~~~D^T=D,\]
for some signature matrix
\[J=\begin{block}{cc}I_r&0\\
0 & -I_{n-r}\end{block}.\]
It follows that 
$\Galpha^T=J \Galpha J$, $(\Gbeta)^T= J \Ggamma J$, $(\Ggamma)^T=J \Gbeta J$,
which implies that if $P$ is a solution of (\ref{RE}), then $\widetilde P=J
P^{-T}J $ is
also a solution and we have $S_{\widetilde P}=S_P^T$.
 If $S_P$ is a real inner extension of $S$ computed as in Theorem
 \ref{thm:GR}, then an analog of Proposition \ref{extension}, in which $P^{-T}$ must be
 replaced by $\widetilde P$, allows
to construct a real symmetric unitary extension of $S$ with the same
properties. However, the situation is more involved than in the complex case.
Even if the algebraic multiplicity of the eigenvalues of  $\nocH$ are all even, 
a symmetric real extension at the
same degree may not exist. We conclude by illustrating this with an example.

\subsection{An example}
Consider the rational matrix  
\[S(s)=\begin{block}{cc} f(s) & 0\\
0 & f(s)\end{block},\]
where $f(s)$ is a strictly proper scalar Schur function of
McMillan degree $1$. Let
\[f(s)= d + c (s-a)^{-1}c,\]
be a symmetric realization of $f(s)$. 
The matrix-valued function $S(s)$ is Schur and contractive at infinity and
has the minimal realization 
\[S(s)=D+C(sI_2-A)^{-1}C^{T},~~~~{\rm with}~~~~
D= d \,I_2,~~C= c\,I_2, ~~A=a\,I_2.\]
In this example, all the 
eigenvalues of $\nocH$ defined in (\ref{def:A})
have even algebraic multiplicity.  We first construct a {\em complex}
symmetric extension at the same degree, degree $2$. 
The Riccati equation associated with $S(s)$ is (see th.3.1)
\begin{equation}
\label{Riccati}
\Ggamma P^2+(\Galpha+\Galpha^*)P+\Gbeta=0,
\end{equation}
where
\[ \Galpha=\frac{ a(1-|d|^2)+c^2\bar d}{1-|d|^2} I_2,~~~~
\Gbeta=\frac{|c|^2}{1-|d|^2} I_2,~~~~
\Ggamma=\frac{|c|^2}{1-|d|^2} I_2.\]
The equation  (\ref{Riccati}) can be rewritten as
\begin{equation}
\label{RiccatiSimp}
P^2+2 \frac{\hbox{Re}(\Galpha)(1-|d|^2)}{|c|^2}P+I_2=0.
\end{equation}
A Hermitian solution to (\ref{RiccatiSimp})  can be diagonalized in the form
$P=U\diag{(\lambda,\mu)}U^*$,
where $U$ is an unitary matrix. 
Note that the eigenvalues of $P$ must be solutions
to the scalar Riccati equation 
\begin{equation}
\label{RiccatiScalar}
p^2+2 \frac{\hbox{Re}(\Galpha)(1-|d|^2)}{|c|^2}p+1=0,
\end{equation}
and must be positive real, since $f(s)$ is Schur (positive real lemma).
 By corollary \ref{prop:main},  a complex symmetric extension is obtained from 
a solution $P$ to (\ref{RiccatiSimp}) which satisfies $PP^T=I_2$.
Such a solution is obtained taking 
\[U=\frac{1}{\sqrt{2}}\begin{block}{cc} 1 & 1\\
-i & i \end{block},\]
and for $\lambda$ and $\mu$  the two distinct solutions to
(\ref{RiccatiScalar}) which satisfy $\lambda\mu=1$.
The solution is 
\[P=\frac{1}{2}\begin{block}{cc}\lambda+\mu & i(\lambda-\mu)\\
-i(\lambda-\mu) & \lambda+\mu\end{block}.\]
It is easy to check that  $P$ is Hermitian and moreover that $PP^T=I_2$.

We now come to the {\em real case} and assume that  $a, c$ and $d$ are real. 
The Schur function $S(s)$ has a real extension if and only if the Riccati
equation (\ref{RiccatiSimp}) has a real solution $P$. Such a real solution 
 must be of the form 
\[P=O\begin{block}{cc} \lambda  & 0\\
0 &\mu
\end{block}O^T,\]
where $O$ is a real orthogonal matrix,
and $\lambda$ and $\mu$ are  (positive real) solutions
to the scalar Riccati equation 
(\ref{RiccatiScalar}). 
Note that $S(s)$ is Schur if and only if (\ref{RiccatiScalar}) has a positive real
solution, which happens if and only if
$a\leq -\frac{c^2}{1-d}$.
A symmetric real extension of $S(s)$ at the same degree is obtained from a
real solution  $P$ to (\ref{RiccatiSimp}) which in addition satisfies
$PP^T=I_2$, but this may only happen  if $P=I_2$, that is  when $a= -c^2/(1-d)$,
in which case the two solutions of (\ref{RiccatiScalar}) are equal to $1$.
For example, the function $f(s)=1/(s+\zeta)$
is Schur if and only if  $\zeta\geq 1$. It
has a real symmetric extension of degree $2$ if and only if  $\zeta=1$.

\end{document}